\documentclass{amsart}

\usepackage{amssymb,amsmath,latexsym,amsthm}
\usepackage{MnSymbol}
\usepackage{graphicx}
\usepackage{fontenc}%
\usepackage{marvosym}
\usepackage{paralist}
\usepackage{color}

\usepackage[verbose]{wrapfig}

\newtheorem{theorem}{Theorem}[section]

\newtheorem{lemma}[theorem]{Lemma}

\usepackage{pstricks,pst-text,pst-grad,pst-node,pst-3dplot,pstricks-add,pst-poly}
\definecolor{pink}{rgb}{1, .75, .8}
\psset{arrows=->, labelsep=3pt, mnode=circle}

\definecolor{lgrey}{gray}{.85}

\newcommand{\norm}[1]{\left\|#1\right\|}  

\begin{document}

\title[Proximal Delaunay Triangulation Regions]{Proximal Delaunay Triangulation Regions}

\author[J.F. Peters]{J.F. Peters}
\email{James.Peters3@umanitoba.ca}
\address{Computational Intelligence Laboratory,
University of Manitoba, WPG, MB, R3T 5V6, Canada}
\thanks{The research has been supported by the Scientific and Technological Research
Council of Turkey (T\"{U}B\.{I}TAK) Scientific Human Resources Development
(BIDEB) under grant no: 2221-1059B211402463 and Natural Sciences \&
Engineering Research Council of Canada (NSERC) discovery grant 185986.}


\date{}

\dedicatory{Dedicated to the Memory of Som Naimpally}

\begin{abstract}
A main result in this paper is the proof that proximal Delaunay triangulation regions are convex polygons.  In addition, it is proved that every Delaunay triangulation region has a local Leader uniform topology.  
\end{abstract}


\maketitle

\section{Introduction}
Delaunay triangulations, introduced by B.N Delone [Delaunay]~\cite{Delaunay1934}, represent pieces of a continuous space.  This representation supports numerical algorithms used to compute properties such as the density of a space.  A \emph{triangulation} is a collection of triangles, including the edges and vertices of the triangles in the collection.  A 2D \emph{Delaunay triangulation} of a set of sites (generators) $S\subset \mathbb{R}^2$ is a triangulation of the points in $S$.  Let $p,q\in S$.  A straight edge connecting $p$ and $q$ is a \emph{Delaunay edge} if and only if the Vorono\"{i} region of $p$~\cite{Edelsbrunner2014,Peters2014arXivVoronoi} and Vorono\"{i} region of $q$ intersect along a common line segment~\cite[\S I.1, p. 3]{Edelsbrunner2001}.  For example, in Fig.~\ref{fig:convexPolygon}, $V_p \cap V_q = \overline{xy}$.  Hence, $\overline{pq}$ is a Delaunay edge in Fig.~\ref{fig:convexPolygon}. 

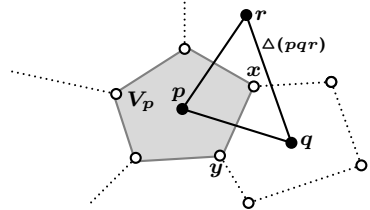
\begin{figure}[!ht]
\begin{center}
 \begin{pspicture}
 (0.0,1.2)(7.5,3.0)
\psframe[linecolor=white](2.5,0.5)(6.5,3.5)
\rput(4.5,2){\PstPentagon[yunit=0.8,fillstyle=solid,fillcolor=lgrey,linestyle=solid,linecolor=gray]}
\pscircle[fillstyle=solid,fillcolor=black](4.5,2.0){0.08}
\psline[linestyle=solid,dotsep=0.03]{-}(4.5,2.0)(5.35,3.25)
\pscircle[fillstyle=solid,fillcolor=black](5.35,3.25){0.08}
\psline[linestyle=solid,dotsep=0.03]{-}(4.5,2.0)(5.95,1.55)
\pscircle[fillstyle=solid,fillcolor=black](5.95,1.55){0.08}
\psline[linestyle=solid]{-}(5.35,3.25)(5.95,1.55)
\psline[linewidth=0.8pt,linestyle=dotted,dotsep=0.05,linecolor=black]{-}(4.53,2.8)(4.53,3.5)
\pscircle[fillstyle=solid,fillcolor=white,linecolor=black](4.53,2.8){0.08}
\psline[linewidth=0.8pt,linestyle=dotted,dotsep=0.05,linecolor=black]{-}(5.45,2.30)(6.48,2.36)
\pscircle[fillstyle=solid,fillcolor=white,linecolor=black](6.48,2.36){0.08}
\pscircle[fillstyle=solid,fillcolor=white,linecolor=black](5.45,2.30){0.08}
\psline[linewidth=0.8pt,linestyle=dotted,dotsep=0.05,linecolor=black]{-}(3.62,2.20)(2.23,2.5)
\pscircle[fillstyle=solid,fillcolor=white,linecolor=black](3.62,2.20){0.08}
\psline[linewidth=0.8pt,linestyle=dotted,dotsep=0.05,linecolor=black]{-}(3.88,1.35)(3.28,0.75)
\pscircle[fillstyle=solid,fillcolor=white,linecolor=black](3.88,1.35){0.08}
\psline[linewidth=0.8pt,linestyle=dotted,dotsep=0.05,linecolor=black]{-}(5.00,1.38)(5.38,0.75)
\psline[linewidth=0.8pt,linestyle=dotted,dotsep=0.05,linecolor=black]{-}(5.38,0.75)(6.88,1.25)
\psline[linewidth=0.8pt,linestyle=dotted,dotsep=0.05,linecolor=black]{-}(6.88,1.25)(6.48,2.36)
\pscircle[fillstyle=solid,fillcolor=white,linecolor=black](5.00,1.38){0.08}
\pscircle[fillstyle=solid,fillcolor=white,linecolor=black](5.38,0.75){0.08}
\pscircle[fillstyle=solid,fillcolor=white,linecolor=black](6.88,1.25){0.08}
\pscircle[fillstyle=solid,fillcolor=white,linecolor=black](6.48,2.36){0.08}
\rput(4.45,2.2){\footnotesize$\boldsymbol{p}$}\rput(6.15,1.55){\footnotesize$\boldsymbol{q}$}
\rput(5.55,3.25){\footnotesize$\boldsymbol{r}$}\rput(5.45,2.50){\footnotesize$\boldsymbol{x}$}
\rput(4.95,1.18){\footnotesize$\boldsymbol{y}$}
\rput(5.98,2.85){\tiny$\boldsymbol{\bigtriangleup(pqr)}$}
\rput(3.95,2.1){\footnotesize$\boldsymbol{V_p}$}
 \end{pspicture}
\end{center}
\caption[]{$p,q\in S,\overline{pq}$ = Delaunay Edge}
\label{fig:convexPolygon}
\end{figure} 

A triangle with vertices $p,q,r\in S$ is a \emph{Delaunay triangle} (denoted $\bigtriangleup(pqr)$ in Fig.~\ref{fig:convexPolygon}), provided the edges in the triangle are Delaunay edges.  
This paper introduces proximal Delaunay triangulation regions derived from the sites of Vorono\"{i} regions~\cite{Peters2014arXivVoronoi}, which are named after the Ukrainian mathematician Georgy Vorono\"{i}~\cite{Voronoi1907}.

A nonempty set $A$ of a space $X$ is a \emph{convex set}, provided $\alpha A + (1-\alpha)A\subset A$ for each $\alpha\in[0,1]$~\cite[\S 1.1, p. 4]{Beer1993}.  
A \emph{simple convex set} is a closed half plane (all points on or on one side of a line in $R^2$~\cite{Edelsbrunner2014}.  

\begin{lemma}\label{lem:convexity}~{\rm \cite[\S 2.1, p. 9]{Edelsbrunner2014}} The intersection of convex sets is convex.
\end{lemma}
\begin{proof}
Let $A,B\subset \mathbb{R}^2$ be convex sets and let $K = A\cap B$.  For every pair points $x,y\in K$, the line segment $\overline{xy}$ connecting $x$ and $y$ belongs to $K$, since this property holds for all points in $A$ and $B$.   Hence, $K$ is convex.
\end{proof}

\section{Preliminaries}
Delaunay triangles are defined on a finite-dimensional normed linear space $E$ that is topological.  For simplicity, $E$ is the Euclidean space $\mathbb{R}^2$.  

\setlength{\intextsep}{0pt}
\begin{wrapfigure}[9]{R}{0.40\textwidth}
\begin{minipage}{4.2 cm}
\begin{center}
 \begin{pspicture}
 (0.0,1.0)(2.5,3.5)
\pscircle[linecolor=black,linestyle=dotted,dotsep=0.05](1.88,2.35){1.08}
\pscircle[fillstyle=solid,fillcolor=white,linecolor=black](1.88,2.35){0.08}
\psline[linestyle=solid,showpoints=true,dotscale=1.0]{-}(0.84,2.35)(2.22,1.35)(2.75,3.00)(0.84,2.35)
\psline[linestyle=dotted,dotsep=0.03]{-}(1.88,2.35)(0.58,1.25)
\psline[linestyle=dotted,dotsep=0.03]{-}(1.88,2.35)(3.50,2.00)
\psline[linestyle=dotted,dotsep=0.03]{-}(1.88,2.35)(1.50,3.50)
\rput(0.58,2.35){\footnotesize$\boldsymbol{p}$}
\rput(2.32,1.20){\footnotesize$\boldsymbol{q}$}
\rput(2.91,3.05){\footnotesize$\boldsymbol{r}$}
\rput(1.88,2.15){\footnotesize$\boldsymbol{u}$}
 \end{pspicture}
\caption{Circumcircle}
\label{fig:circumcircle}
\end{center}
\end{minipage}
\end{wrapfigure}
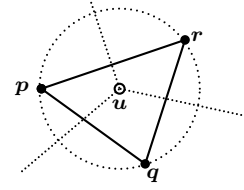
\setlength{\intextsep}{2pt}

The \emph{closure} of $A\subset E$ (denoted $\mbox{cl}A$) is defined by $\mbox{cl}A=\left\{x\in E: inf\norm{x-a}=0, a\in A\right\}$.  Let $A^c$ be complement of $A$ (all points of $E$ not in $A$).  The \emph{boundary} of $A$ (denoted $\mbox{bdy}A$) is the set of all points in $A$ and $A^c$.  An important structure is the \emph{interior} of $A$ (denoted $\mbox{int}A$), defined by $\mbox{int}A = \mbox{cl}A - \mbox{bdy}A$.  For example, the interior of a Delaunay edge $\overline{pq}$ are all of the points in the segment, except the endpoints $p$ and $q$.

The space $E$ endowed with a proximity relation $\delta$ (near) and its counterpart $\underline{\delta}$ (far) facilitates the description of properties of Delaunay edges, triangles, triangulations and regions.  Let $A,B\subset E$.  The set $A$ is near $B$ (denoted $A\ \delta\ B$), provided $\mbox{cl}A\cap \mbox{cl}B\neq \emptyset$~\cite{Concilio2009}, which is central in near set theory~\cite{Naimpally2013,Peters2012ams}.  Sets $A,B$ are \emph{far} apart (denoted $A\ \underline{\delta}\ B$), provided $\mbox{cl}A\cap \mbox{cl}B = \emptyset$. For example, Delaunay edges $\overline{pq}\ \delta\ \overline{qr}$ are near, since the edges have a common point, {\em i.e.}, $q\in \overline{pq}\cap \overline{qr}$, {\em e.g.}, $\overline{pq}\ \delta\ \overline{qr}$ in Fig.~\ref{fig:convexPolygon}.  By contrast, $\overline{pr}\ \underline{\delta}\ \overline{xy}$ in Fig.~\ref{fig:convexPolygon}.


For a Delaunay triangle $\bigtriangleup(pqr)$, a \emph{circumcircle} passes through the vertices $p,q,r$ of the triangle (see Fig.~\ref{fig:circumcircle} for an example).  The center of a circumcircle $u$ is the Vorono\"{i} vertex at the intersection of three Vorono\"{i} regions, {\em i.e.}, $u = V_p\cap V_q\cap V_r$. The circumcircle radius $\rho = \norm{u-p} = \norm{u-q} = \norm{u-r}$~\cite[\S I.1, p. 4]{Edelsbrunner2001}, which is the case in Fig.~\ref{fig:circumcircle}. 

\begin{lemma}\label{lem:circumcircle}
Let circumcircle $\bigcirc(pqr)$ pass through the vertices of a Delaunay triangle $\bigtriangleup(pqr)$, then the following statements are equivalent.
\begin{compactenum}[1$^o$]
\item The center $u$ of $\bigcirc(pqr)$ is a vertex common to Vorono\"{i} regions $V_p,V_q,V_r$.
\item $u = \mbox{cl}V_p\cap \mbox{cl}V_q\cap \mbox{cl}V_r$.
\item $V_p\ \delta\ V_q\ \delta\ V_r$. 
\end{compactenum}
\end{lemma}
\begin{proof}$\mbox{}$\\
1$^o\Leftrightarrow 2^o\Leftrightarrow 3^o$.
\end{proof}

\begin{theorem}\label{lem:DelaunayTriangle}
A triangle $\bigtriangleup(pqr)$ is a Delaunay triangle if and only if the center of the circumcircle $\bigcirc(pqr)$ is the vertex common to three Vorono\"{i} regions.
\end{theorem}
\begin{proof}
The circle $\bigcirc(pqr)$ has center $u = \mbox{cl}V_p\cap \mbox{cl}V_q\cap \mbox{cl}V_r$ (Lemma~\ref{lem:circumcircle}) $\Leftrightarrow\\ 
\bigcirc(pqr)$ center is the vertex common to three Vorono\"{i} regions $V_p,V_q,V_r$  
$\Leftrightarrow\\ 
\overline{pq}, \overline{pr}, \overline{qr}$ are Delaunay edges $\Leftrightarrow$\\ 
$\bigtriangleup(pqr)$ is a Delaunay triangle.
\end{proof}

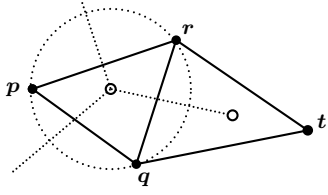
\begin{figure}[!ht]
\begin{center}
 \begin{pspicture}
 (0.0,1.5)(7.5,3.2)
\pscircle[linecolor=black,linestyle=dotted,dotsep=0.05](3.88,2.35){1.08}
\pscircle[fillstyle=solid,fillcolor=white,linecolor=black](3.88,2.35){0.08}
\psline[linestyle=solid,showpoints=true,dotscale=1.0]{-}(2.84,2.35)(4.22,1.35)(4.75,3.00)(2.84,2.35)
\psline[linestyle=dotted,dotsep=0.03]{-}(3.88,2.35)(2.58,1.25)
\psline[linestyle=dotted,dotsep=0.03]{-}(3.88,2.35)(5.50,2.00)
\psline[linestyle=dotted,dotsep=0.03]{-}(3.88,2.35)(3.50,3.50)
\pscircle[fillstyle=solid,fillcolor=black,linecolor=black](6.50,1.80){0.07}
\pscircle[fillstyle=solid,fillcolor=white,linecolor=black](5.50,2.00){0.08}
\psline[linestyle=solid]{-}(4.75,3.00)(6.50,1.80)
\psline[linestyle=solid]{-}(4.22,1.35)(6.50,1.80)
\rput(2.58,2.35){\footnotesize$\boldsymbol{p}$}
\rput(4.32,1.15){\footnotesize$\boldsymbol{q}$}
\rput(4.91,3.15){\footnotesize$\boldsymbol{r}$}
\rput(6.69,1.93){\footnotesize$\boldsymbol{t}$}
 \end{pspicture}
\end{center}
\caption[]{Strongly near Delaunay triangles $\bigtriangleup(pqr)$ and $\bigtriangleup(qrt)$}
\label{fig:nearTriangles}
\end{figure}

\section{Main Results}
Vorono\"{i} regions $V_p,V_q$ are strongly near (denoted $V_p\ \mathop{\delta}\limits^{\doublewedge}\ V_q$) if and only if the regions have a common edge.  For example, $V_p\ \mathop{\delta}\limits^{\doublewedge}\ V_q$ in Fig.~\ref{fig:convexPolygon}.  \emph{Strongly near} Delaunay triangles have a common edge.  Delaunay triangles $\bigtriangleup(pqr)$ and $\bigtriangleup(qrt)$ are strongly near in Fig.~\ref{fig:nearTriangles}, since edge $\overline{qr}$ is common to both triangles.  In that case, we write $\bigtriangleup(pqr)\ \mathop{\delta}\limits^{\doublewedge}\ \bigtriangleup(qrt)$.

\begin{theorem}
The following statements are equivalent.
\begin{compactenum}[1$^o$]
\item $\bigtriangleup(pqr)$ is a Delaunay triangle.
\item Circumcircle $\bigcirc(pqr)$ has center $u = \mbox{cl}V_p\cap \mbox{cl}V_q\cap \mbox{cl}V_r$.
\item $V_p\ \mathop{\delta}\limits^{\doublewedge}\ V_q\ \mathop{\delta}\limits^{\doublewedge}\  V_r$.
\item $\bigtriangleup(pqr)$ is the union of convex sets.
\end{compactenum}
\end{theorem}
\begin{proof}$\mbox{}$\\
1$^o\Leftrightarrow 2^o$, since the center of the circumcircle $\bigcirc(pqr)$ is the vertex common to three Vorono\"{i} regions (from Theorem~\ref{lem:DelaunayTriangle}),\\
$\Leftrightarrow 3^o$ ($V_p,V_q,V_r$ are strongly near),\\
$\Leftrightarrow$ each pair of Delaunay edges have a common vertex, which is a convex set (from Lemma~\ref{lem:convexity}), since $\overline{pq}, \overline{pr}, \overline{qr}$ are convex sets. 
\end{proof}

Let $S\subset \mathbb{R}^2$ be a finite set of sites, $L$ a finite set of line segments, each connecting a pair of points in $S$.  A \emph{constrained triangulation} of $S$ and $L$ is a triangulation of $S$ that contains all line segments of $L$ as edges.  By adding straight edges connecting points in $S$ so that no new edge has a point in common with previous edges, we construct a constrained triangulation.  Let $P$ be a polygon.  Two points $p,q\in P$ are \emph{visible}, provided the line segment $\overline{pq}$ is in $\mbox{int}P$~\cite{Ghosh2007}.  

Let $p,q\in S, \overline{pq}\in L$.  Points $p,q$ are visible from each other, which implies that $\overline{xy}$ contains no point of $S-\left\{p,q\right\}$ in its interior and $\overline{xy}$ shares no interior point with a constraining line segment in $L-\overline{pq}$~\cite[\S II, p. 32]{Edelsbrunner2001}. 

\begin{theorem}
If points in $\mbox{int}\ \overline{pq}$ are visible from $p,q$, then $\mbox{int}\ \overline{pq}\ \underline{\delta}\ S-\left\{p,q\right\}$ and $\overline{pq}\ \underline{\delta}\ \overline{xy}\in L-\overline{pq}$ for all $x,y\in S-\left\{p,q\right\}$.
\end{theorem}
\begin{proof}
By definition, points $p,q\in S$ are visible, provided no points in the interior $\overline{pq}$ are in $S-\left\{p,q\right\}$.  Then $\mbox{int}\ \overline{xy}$ is far from $S-\left\{p,q\right\}$ ($\mbox{int}\ \overline{xy}\ \underline{\delta}\ S-\left\{p,q\right\}$).  In addition, no line segments $\overline{xy}\in L-\overline{pq}$ have points in common with any points in the interior of $\overline{pq}$ ($\mbox{int}\ \overline{pq}\ \underline{\delta}\ \overline{xy}$ for all line segments $\overline{xy}\in L-\overline{pq}$).
\end{proof}

 Let $K$ be a set of straight edges in a constrained triangulation of $S$ and $L$.  An edge $\overline{pq}, p,q\in S$ belongs to a constrained Delaunay triangulation of $S$ and $L$, provided
\begin{compactenum}[1$^o$]
\item Edge $\overline{pq}\in L, p,q\in S$, or
\item $p,q$ are visible from each other and there is a circle passing through $p,q$ so that each point inside the circle
is invisible from every point $x\in \mbox{int}\ \overline{pq}$. 
\end{compactenum}

\begin{lemma}\label{lem:constrained}~{\rm \cite[\S II.5, p. 32]{Edelsbrunner2001}} If every edge of $K$ is locally Delaunay, then $K$ is the constrained triangulation of $S$ and $L$.
\end{lemma}
 
A \emph{Delaunay triangulation region} $\mathcal{D}$ is a collection of Delaunay triangles such that every pair triangles in the collection is strongly near.  That is, every Delaunay triangulation region is a triangulation of a finite set of points and the triangles in the region are pairwise strongly near.  \emph{Proximal Delaunay triangulation regions} have at least one vertex in common. 

\begin{lemma}\label{lem:convexPolygon}
A Delaunay triangulation region is a convex polygon.  
\end{lemma}
\begin{proof}
Immediate from Lemma~\ref{lem:convexity}.
\end{proof}

\begin{theorem}
Proximal Delaunay triangulation regions are convex polygons.
\end{theorem}
\begin{proof}
Let $\mathcal{D}$ be a proximal Delaunay triangulation region.  By definition, $\mathcal{D}$ is the nonempty intersection of convex sets.  From Lemma~\ref{lem:convexity}, $\mathcal{D}$ is convex.  Hence, from Lemma~\ref{lem:convexPolygon}, $\mathcal{D}$ is a convex polygon.
\end{proof}

A \emph{local Leader uniform topology} on a set in the plane is determined by finding those sets that are close to each given set.

\begin{theorem}
Every Delaunay triangulation region has a local Leader uniform topology {\rm (application of~{\rm \cite{Leader1959}})}.
\end{theorem}
\begin{proof} Let $A\in \mathcal{D}$, a set of strongly near Delaunay triangles in Delaunay triangulation $\mathcal{D}$.
Find all $B\in \mathcal{D}$ that are close to $A$, {\em i.e.}, $A\ \delta\ B$.  For each $A$, this procedure determines a family of Delaunay triangles that are near $A$.  By definition, this procedure also determines a local Leader uniform topology.
\end{proof}


\end{document}